\documentclass[a4paper,12pt]{amsart}
\usepackage{graphicx}
\usepackage{mathtools}
\usepackage{amsmath}
\usepackage{amssymb}
\usepackage{amsopn}
\usepackage{epsfig}
\usepackage{amsfonts}
\usepackage{latexsym}
\usepackage[T1]{fontenc}    %sajat
\usepackage[utf8]{inputenc}%sajat
\usepackage{enumerate}
\usepackage{amssymb}
\usepackage{fullpage}
\usepackage{color}
\usepackage{mathtools}
\def\RR{\mathbb R}
\def\NN{\mathbb N}
\def\uu{\mathcal U}

\def\ff{\mathcal F}

\def\ss{\mathcal S}

\newcommand{\set}[1]{\left\lbrace #1\right\rbrace}%set
\providecommand{\abs}[1]{\left\lvert#1\right\rvert}

\newcommand{\remove}[1]{ }
\newcommand{\qtq}[1]{\quad \text{#1}\quad }%\quad\text{and}\quad

\newtheorem{theorem}{Theorem}[section]
\newtheorem{proposition}[theorem]{Proposition}
\newtheorem{lemma}[theorem]{Lemma}
\newtheorem{corollary}[theorem]{Corollary}

\theoremstyle{definition}

\theoremstyle{remark}
\newtheorem{remark}[theorem]{Remark}

\newtheorem{example}[theorem]{Example}

\numberwithin{equation}{section}
\numberwithin{figure}{section}
\begin{document}
\title{Expansions in multiple bases over general alphabets}
\author{Yuru Zou}
\address{College of Mathematics and Statistics,
Shenzhen University,
Shenzhen 518060,
People’s Republic of China}
\email{yuruzou@szu.edu.cn}

\author{Vilmos Komornik}
\address{College of Mathematics and Statistics,
Shenzhen University,
Shenzhen 518060,
People’s Republic of China,
and
Département de mathématique,
Université de Strasbourg,
7 rue René Descartes,
67084 Strasbourg Cedex, France}
\email{komornik@math.unistra.fr}
\author{Jian Lu}
\address{College of Mathematics and Statistics,
Shenzhen University,
Shenzhen 518060,
People’s Republic of China}
\email{jianlu@szu.edu.cn}

%\curraddr{}
\date{Version 2021-02-19}

%\thanks{*Corresponding author; email address: komornik@math.unistra.fr.}
\thanks{This work was supported by the National Natural Science Foundation
of China (NSFC)  \#11871348, \#61972265.
}

\begin{abstract}\mbox{}
Expansions in non-integer bases have been extensively investigated since a pioneering work of R\'enyi.
We introduce a
more general framework of alphabet-base systems that also includes Pedicini's general alphabets and the multiple-base expansions of Neunhäuserer and Li.
We extend the Parry type lexicographic theory to this setup,
and we improve and generalize various former results on unique expansions.
\end{abstract}
\maketitle

\section{Introduction}\label{s1}

Since the pioneering work of R\'enyi \cite{Renyi1957} many papers have been devoted to the study of expansions of the form
\begin{equation*}
x=\sum_{i=1}^{\infty}\frac{c_i}{q^i},\quad (c_i)\in\set{0,1,\ldots,M}^{\NN}
\end{equation*}
in some real \emph{base} $q>1$ over the \emph{alphabet} $\set{0,1,\ldots,M}$, where $M$ is a given positive integer.
See, e.g.,
\cite{
DK2009,
DKL2016,
K2011,
KKL2017,
KL2007,
KLZ2017,
ZK2015,
ZWLB2016,
ZLLK2020}
and their references.
One of the main tools in these investigations is the lexicographic characterization of specific expansions; this fruitful idea was introduced by Parry \cite{Parry1960}.
A new impetus was given by the discovery of surprising unique expansions by Erd\H os et al. \cite{EHJ1991}.
A number of results have been extended to more general alphabets by Pedicini \cite{Ped2005}, and recently to expansions in multiple bases by Neunhäuserer 
\cite{Neunhauserer2019,
Neunhauserer2019b}
and Li \cite{Li2019}.
The purpose of this paper is to introduce a broad generalization of these frameworks, and to improve several theorems on multiple-base expansions.

We fix a finite \emph{alphabet--base system} $S:=\set{(d_0,q_0),\ldots,(d_M,q_M)}$ of couples of real numbers with $q_0,\ldots,q_M>1$, and
we  consider series of the form
\begin{equation*}
\pi((j_i))
=\pi(j_1j_2\cdots)
:=\sum_{i=1}^{\infty}\frac{d_{j_i}}{q_{j_1}\cdots q_{j_i}},
\end{equation*}
corresponding to the sequences $(j_i)=j_1j_2\cdots\in\set{0,\ldots,M}^{\NN}$.
For example,
\begin{equation}\label{11}
\pi(j^\infty)=\frac{d_j}{q_j-1},\quad j=0,\cdots, M;
\end{equation}
here $j^\infty$ denotes the constant sequence $jj\cdots$.
Henceforth by a \emph{sequence} we always mean an element of $\set{0,\ldots,M}^{\NN}$.
A sequence $(j_i)$ is called an \emph{expansion} of a real number $x$ if
$\pi((j_i))=x$.

Our setup reduces to the case of \cite{Ped2005} if $q_1=\cdots=q_M$,  to the framework of \cite{Neunhauserer2019,Li2019} if $\set{d_0,\ldots,d_M}=\set{0,\ldots,M}$, and to the classical situation if both conditions are satisfied.

In order to obtain deeper results we restrict ourselves to \emph{regular} alphabet--base systems, i.e., setting
\begin{equation}\label{12}
\lambda:=\min_j\pi(j^\infty)
\qtq{and}
\Lambda:=\max_j\pi(j^\infty),
\end{equation}
we assume  the following properties:
\begin{align}
&\frac{d_0+\lambda}{q_0}
<\frac{d_1+\lambda}{q_1}<\cdots<
\frac{d_M+\lambda}{q_M};\label{13}\\
&\frac{d_0+\Lambda}{q_0}<\frac{d_1+\Lambda}{q_1}
<\cdots<
\frac{d_M+\Lambda}{q_M};\label{14}\\
&\frac{d_j+\Lambda}{q_j}
\ge\frac{d_{j+1}+\lambda}{q_{j+1}}
\qtq{for all}j<M.\label{15}
\end{align}
They are satisfied by the \emph{classical system}
\begin{equation}\label{16qqq}
S_{M, q}:=\set{(0, q), \ldots, (M, q)}
\qtq{with}q\in(1, M+1],
\end{equation}
and they reduce to the ``Pedicini condition'' in the more general  equal base case.

First we generalize the $\beta$-expansions of R\'enyi \cite{Renyi1957}.
We introduce the usual lexicographic order between sequences  $(j_i), (k_i)$: we write $(j_i)\prec (k_i)$ or $(k_i)\succ (j_i)$ if there exists an index $n$ such that $j_i=k_i$ for all $i<n$, and $j_n<k_n$, and we write $(j_i)\preceq (k_i)$ or $(k_i)\succeq (j_i)$ if either $(j_i)\prec (k_i)$ or $(j_i)=(k_i)$.
A sequence $(j_i)$ is called  \emph{finite} if there exists a last index $j_n>0$, and \emph{infinite} otherwise.
Similarly, a sequence $(j_i)$ is called  \emph{co-finite} if there exists a last index $j_n<M$, and \emph{co-infinite} otherwise.

\begin{theorem}\label{t11}
Let $S$ be  a regular alphabet--base system.
\begin{enumerate}[\upshape (i)]
\item If $x\notin[\lambda,\Lambda]$, then $x$ has no expansion.
\item If $x\in[\lambda,\Lambda]$ then $x$ has a lexicographically largest expansion, a lexicographically largest infinite expansion, a lexicographically smallest expansion, and a lexicographically smallest co-infinite expansion.
\end{enumerate}
\end{theorem}
The expansions in Theorem \ref{t11} are called the \emph{greedy, quasi-greedy, lazy} and \emph{quasi-lazy} expansions of $x$, respectively, and they are denoted by $b(x), a(x), l(x)$ and $m(x)$.

Next we generalize a theorem of Parry \cite{Parry1960} on the lexicographic characterization of $\beta$-expansions.
Let us introduce the following special sequences:
\begin{align}
&\alpha^j
:=a\left(q_j\left(\frac{d_{j+1}+\lambda}{q_{j+1}}-\frac{d_j}{q_j}\right)\right),\quad j=0,\ldots,M-1,\label{16}\\
&\gamma^j
:=m\left(
q_{j}\left(\frac{d_{j-1}+\Lambda}{q_{j-1}}-\frac{d_{j}}{q_{j}}\right)
\right),\quad j=1,\ldots,M.\label{17}
\end{align}
They are well defined by the regularity of $S$: see Lemma \ref{l31} below.
The following theorem  answers a question of Li \cite{Li2019}:

\begin{theorem}\label{t12}
Let $S$ be  a regular alphabet--base system.
\begin{enumerate}[\upshape (i)]
\item The \emph{greedy map} $x\mapsto b(x)$ is an increasing bijection of $[\lambda,\Lambda]$ onto the set of all sequences $(j_i)$ satisfying the following condition:
\begin{equation}\label{18}
(j_{n+i})\prec \alpha^{j_n}\qtq{whenever}j_n<M.
\end{equation}
\item The \emph{quasi-greedy map} $x\mapsto a(x)$ is an increasing bijection of $[\lambda,\Lambda]$ onto the set of all infinite sequences $(j_i)$ satisfying the following condition:
\begin{equation*}
(j_{n+i})\preceq\alpha^{j_n}\qtq{whenever}j_n<M.
\end{equation*}
\item The \emph{lazy map} $x\mapsto l(x)$ is an increasing bijection of $[\lambda,\Lambda]$ onto the set of all sequences $(j_i)$ satisfying the following  condition:
\begin{equation}\label{110}
(j_{n+i})\succ\gamma^{j_n}
\qtq{whenever}j_n>0.
\end{equation}
\item The \emph{quasi-lazy map} $x\mapsto m(x)$ is an increasing bijection of $[\lambda,\Lambda]$ onto the set of all co-infinite sequences $(j_i)$ satisfying the following condition:
\begin{equation*}
(j_{n+i})\succeq \gamma^{j_n}
\qtq{whenever}j_n>0.
\end{equation*}
\end{enumerate}
\end{theorem}

Since an expansion is unique if and only if it is greedy and lazy at the same time, Theorem  \ref{t12} implies the following corollary:

\begin{corollary}\label{c13}\mbox{}
Let $S$ be  a regular alphabet--base system.
A expansion $(j_i)$ of some number $x\in[\lambda, \Lambda]$  is unique if and only if the conditions \eqref{18} and \eqref{110}  are satisfied.
\end{corollary}

Using Corollary \ref{c13} we may estimate the number of unique expansions.
Given a regular alphabet-base system $S$, we denote by $\uu_S$ the set of numbers $x\in[\lambda, \Lambda]$ having a unique expansion, and by $\uu_S'$ the set of corresponding expansions.
The set $\uu_S$ always contains $\lambda$ and $\Lambda$ by Lemma \ref{l24} below; the corresponding elements $0^{\infty}$ and $M^{\infty}$ of  $\uu_S'$ are called \emph{trivial unique expansions}.

Henceforth we frequently use the \emph{reflection}  $\overline{(j_i)}=(\overline{j_i})$ of a given sequence $(j_i)$, defined by the formula $\overline{j_i}:=M-j_i$.

\begin{example}\label{e14}
Let us consider the classical system \eqref{16qqq}.
We have $\lambda=0$, $\Lambda=\frac{M}{q-1}$,
\begin{equation*}
q_j\left(\frac{d_{j+1}+\lambda}{q_{j+1}}-\frac{d_j}{q_j}\right)=1\qtq{for all}j<M,
\end{equation*}
and
\begin{equation*}
q_{j}\left(\frac{d_{j-1}+\Lambda}{q_{j-1}}-\frac{d_{j}}{q_{j}}\right)=\frac{M}{q-1}-1 \qtq{for all} j>0.
\end{equation*}
Therefore, denoting by $\alpha_{M,q}$ and $\gamma_{M,q}$  the quasi-greedy expansion of $1$ and  the quasi-lazy expansion of $\frac{M}{q-1}-1$ in base $q$, respectively, we infer from Corollary \ref{c13} that
\begin{equation*}
\uu_{S_{M,q}}'=\set{(j_i):(j_{n+i})<\alpha_{M,q}\text{ whenever }j_n<M,\text{ and } (j_{n+i})>\gamma_{M,q}\text{ whenever }j_n>0}.
\end{equation*}
Furthermore, $\gamma_{M,q}=\overline{\alpha_{M,q}}$ by a well-known variant of Lemma \ref{l23} below.
\end{example}

Baker \cite{B2014} has proved that $\uu'_{S_{M, q}}$ contains only the two trivial unique expansions if and only if $q\le q_{GR}$, where
\begin{equation*}
q_{GR}=
\begin{cases}
(k+\sqrt{k^2+4k})/2&\text{if $M=2k-1$,}\\
k+1&\text{if $M=2k$.}
\end{cases}
\end{equation*}
The notation is motivated by the special case $M=1$ where $q_{GR}$ is equal to the \emph{Golden Ratio}; see Erd\H os et al. \cite{EHJ1991}.
The quasi-greedy expansion of $1$ in base $q_{GR}$ is given by the following formula:
\begin{equation}\label{112}
\alpha_{GR}:=
\begin{cases}
(k(k-1))^\infty &\text{if $M=2k-1$,}\\
k^{\infty}&\text{if $M=2k$.}
\end{cases}
\end{equation}

There exists another constant  $q_{KL}$ such that $\uu'_{S_{M, q}}$ is countable if and only if $q<q_{KL}$.
It was proved by Glendinning and Sidorov \cite{GS2001} for $M=1$, and then in \cite{DK2009} for all $M\ge 1$ that $q_{KL}$ is the \emph{Komornik--Loreti constant},  \cite{KL1998,KL2002}, i.e., the quasi-greedy expansion of $1$ in base $q_{KL}$ is given by the formula
\begin{equation}\label{113}
\alpha_{KL}:=
\begin{cases}
(k-1+\tau_i) &\text{if $M=2k-1$,}\\
(k+\tau_i-\tau_{i-1})&\text{if $M=2k$,}
\end{cases}
\end{equation}
where $(\tau_i)=1101\ 0011\ \cdots$ denotes the truncated Thue--Morse sequence.

Using the sequences $\alpha_{GR}$ and $\alpha_{KL}$ we may generalize Example \ref{e14}:

\begin{theorem}\label{t15}\mbox{}
Let $S$ be  a regular alphabet--base system, and consider the sequences \eqref{16}, \eqref{17}, \eqref{112} and \eqref{113}.
\begin{enumerate}[\upshape (i)]
\item If $\alpha^j\preceq \alpha_{GR}$ and $\gamma^{M-j}\succeq\overline{\alpha_{GR}}$ for all $j<M$, then $\uu_S'=\set{0^{\infty},M^{\infty}}$;
\item If $\alpha^j\succ\alpha_{GR}$ and $\gamma^{M-j}\prec\overline{\alpha_{GR}}$ for all $j<M$, then $\uu'_{S}$ is an infinite set;
\item If $\alpha^j\prec \alpha_{KL}$ and $\gamma^{M-j}\succ\overline{\alpha_{KL}}$ for all $j<M$, then $\uu_{S}'$ is countable;
\item If $\alpha^j\succeq\alpha_{KL}$ and $\gamma^{M-j}\preceq\overline{\alpha_{KL}}$ for all $j<M$, then $\uu'_{S}$ has the power of the continuum.
\end{enumerate}
\end{theorem}

For regular systems having two elements we have more precise results:

\begin{theorem}\label{t16}
Let $S=\set{(d_0, q_0), (d_1, q_1)}$ be a regular alphabet-base system  with $q_0, q_1\in(1,2]$.
Then the following three properties are equivalent:
\begin{enumerate}[\upshape (i)]
\item $\uu_S'$ is an infinite set;
\item $\alpha^0\succ (10)^\infty$ and $\gamma^1\prec (01)^\infty$;
\item $q_0>1+\frac{1}{q_1}$ and $q_1>1+\frac{1}{q_0}$.
\end{enumerate}
Furthermore, if $\uu_S'$ is not infinite, then $\uu_S'=\set{0^{\infty},M^{\infty}}$.
\end{theorem}
Theorem \ref{t16} applies to every system $S=\set{(0, q_0), (1, q_1)}$ with $q_0, q_1\in(1,2]$; see Lemma \ref{l41} below.

For special alphabets there is a variant of property (iii) of Theorem \ref{t16}:

\begin{theorem}\label{t17}
Consider an alphabet-base system $S=\set{(0, q_0), (1, q_1)}$ with $d_1\ge 0\ge d_0$ and $q_0, q_1\in(1,2]$.
\begin{enumerate}[\upshape (i)]
\item If $q_1\ge q_0$, then
\begin{equation*}
\uu_S'\text{ is  infinite }
\Longleftrightarrow
q_0>1+\frac{1}{q_1}.
\end{equation*}
\item If $q_0\ge q_1$, then
\begin{equation*}
\uu_S'\text{ is infinite }
\Longleftrightarrow
q_1>1+\frac{1}{q_0}.
\end{equation*}
\end{enumerate}
\end{theorem}

Theorem \ref{t17} improves a former theorem of  Neunhäuserer \cite{Neunhauserer2019}: he assumed that $d_0=0$, $d_1=1$, $q_1\ge q_0$, and he did not investigate the case where
\begin{equation*}
1+\frac{q_0}{q^2_1}
\le  q_0
\le 1+\frac{1}{q_1}.
\end{equation*}

The paper is arranged as follows.
Theorem \ref{t11} is proved in Section \ref{s2}, Theorems \ref{t12} and \ref{t15} are proved  in Section \ref{s3}.
Finally, the proofs of Theorems \ref{t16} and \ref{t17} are given in Section \ref{s4}.

\section{Proof of Theorem \ref{t11}}\label{s2}

First we establish some results on  general alphabet-base systems
\begin{equation*}
S:=\set{(d_0,q_0),\ldots,(d_M,q_M)}
\qtq{with}
\lambda:=\min_j\pi(j^\infty)
\qtq{and}
\Lambda:=\max_j\pi(j^\infty).
\end{equation*}
We recall that $\set{0,\ldots,M}^{\NN}$ is a compact metrizable space for the Tychonoff product topology, and that the convergence in this space is the component-wise convergence.

\begin{lemma}\label{l21}\mbox{}
\begin{enumerate}[\upshape (i)]
\item The series
\begin{equation*}
\pi((j_i))=\sum_{i=1}^{\infty}\frac{d_{j_i}}{q_{j_1}\cdots q_{j_i}}
\end{equation*}
converges for every sequence $(j_i)$, and the function $\pi:\set{0,\ldots,M}^{\NN}\to\RR$ is continuous.
\item The set
\begin{equation*}
J:=\set{\pi((j_i))\ :\ (j_i)\in\set{0,\ldots,M}^{\NN}}
\end{equation*}
is compact,  and
\begin{equation}\label{22}
\lambda=\min J\quad
\Lambda=\max J.
\end{equation}
\end{enumerate}
\end{lemma}

\begin{proof}
(i) Setting
\begin{equation*}
d:=\max\set{\abs{d_0},\ldots,\abs{d_M}}
\qtq{and}
q:=\min\set{q_0,\ldots,q_M}
\end{equation*}
we have $q>1$, and therefore
\begin{equation}\label{23}
\sum_{i=n+1}^{\infty}\abs{\frac{d_{j_i}}{q_{j_1}\cdots q_{j_i}}}
\le\frac{d}{q^n(q-1)}
\to 0\qtq{as}n\to\infty.
\end{equation}
The continuity of $\pi$ follows by observing that if $(j_i), (j_i')\in\set{0,\ldots,M}^{\NN}$ satisfy $j_i=j_i'$ for $i=1,\ldots, n$, then
\begin{equation*}
\abs{\pi((j_i))-\pi((j_i'))}
\le\sum_{i=n+1}^{\infty}\abs{\frac{d_{j_i}-d_{j_i'}}{q_{j_1}\cdots q_{j_i}}}\le\frac{2d}{q^n(q-1)},
\end{equation*}
and the right hand side tends to zero as $n\to\infty$.
\medskip

(ii) $J$ is compact because it is the continuous image of a compact space by (i).
Furthermore, it contains $\lambda$ and $\Lambda$ because $\pi(j^{\infty})\in J$ for all $j$.

It remains to prove that
\begin{equation}\label{24}
\lambda\le\pi((j_i))\le\Lambda
\end{equation}
for every sequence $(j_i)$.
Since
\begin{equation*}
\pi((j_i))=\lim_{n\to\infty}\pi(j_1\cdots j_n0^{\infty})
\end{equation*}
by \eqref{23}, it suffices to prove \eqref{24} for sequences ending with $0^{\infty}$.
Proceeding by induction on $n\ge 1$, we prove  \eqref{24} for all sequences satisfying $j_i=0$ for all $i\ge n$.
The case $n=1$ follows from the definition of $\lambda$ and $\Lambda$:
\begin{equation*}
\lambda=\min_j\pi(j^\infty)
\le \pi(0^\infty)
\le\max_j\pi(j^\infty)=\Lambda.
\end{equation*}

Assume that the property holds for some $n\ge 1$, and consider a sequence satisfying $j_i=0$ for all $i\ge n+1$.
Using the induction hypothesis hence we get
\begin{align*}
&\pi(j_1j_2\cdots)
=\frac{d_{j_1}+\pi(j_2j_3\cdots)}{q_{j_1}}
\ge \frac{d_{j_1}+\lambda }{q_{j_1}}
\intertext{and}
&\pi(j_1j_2\cdots)
=\frac{d_{j_1}+\pi(j_2j_3\cdots)}{q_{j_1}}
\le \frac{d_{j_1}+\Lambda }{q_{j_1}}.
\end{align*}
We conclude by observing that
\begin{align*}
&\lambda\le\pi(j_1^{\infty})=\frac{d_{j_1}}{q_{j_1}-1}\Longrightarrow \frac{d_{j_1}+\lambda }{q_{j_1}}\ge\lambda
\intertext{and}
&\Lambda\ge\pi(j_1^{\infty})=\frac{d_{j_1}}{q_{j_1}-1}\Longrightarrow \frac{d_{j_1}+\Lambda }{q_{j_1}}\le\Lambda.\qedhere
\end{align*}
\end{proof}

We say that a sequence $(j_i)\in\set{0,\ldots,M}^{\NN}$ is  a \emph{subexpansion} or a \emph{superexpansion} of $x$ if
\begin{equation*}\label{25}
\pi((j_i))\le x\qtq{or}
\pi((j_i))\ge x,
\end{equation*}
respectively.

\begin{lemma}\label{l22}\mbox{}
\begin{enumerate}[\upshape (i)]
\item Every $x\ge\lambda $ has a lexicographically largest subexpansion $b(x)$.
Furthermore, the map $x\mapsto b(x)$ is  non-decreasing on $[\lambda ,\infty)$, and $b(x)=M^{\infty}$ for every $x\ge\Lambda$.
\item Every $x\ge\lambda $ has a lexicographically largest infinite subexpansion $a(x)$.
Furthermore, the map $x\mapsto a(x)$ is  non-decreasing on $[\lambda ,\infty)$, and $a(x)=M^{\infty}$ for every $x\ge\Lambda$.
\item Every $x\le\Lambda $ has a lexicographically smallest superexpansion $l(x)$.
Furthermore, the map $x\mapsto l(x)$ is  non-decreasing on $(-\infty,\Lambda]$, and $l(x)=0^{\infty}$ for every $x\le\lambda$.
\item Every $x\le\Lambda $ has a lexicographically smallest co-infinite superexpansion $m(x)$.
Furthermore, the map $x\mapsto m(x)$ is  non-decreasing on $(-\infty,\Lambda]$, and $m(x)=0^{\infty}$ for every $x\le\lambda$.
\end{enumerate}
\end{lemma}

\begin{proof}
(i) Write $\lambda=\pi(j^{\infty})$, and
fix an arbitrary $x\ge\lambda $.
Then $j^{\infty}$ is a subexpansion of $x$, so that the subexpansions of $x$ for a non-empty set $\ff(x)$.
Let $j_1$ be the largest index such that  at least one of these sequences starts with $j_1$, and let $\ff_1(x)$ be the non-empty family of sequences in $\ff(x)$ that start with $j_1$.
Proceeding by induction, if $j_1,\ldots, j_n$ and $\ff_n(x)$ have  already been defined for some $n\ge 1$, then let $j_{n+1}$ be the largest index such that  at least one of the sequences in $\ff_{n}(x)$ starts with $j_1\cdots j_{n+1}$, and let $\ff_{n+1}(x)$ be the non-empty family of sequences in $\ff_n(x)$ that start with $j_1\cdots j_{n+1}$.
We have $\pi((j_i))\le x$ by the continuity of the map $\pi$, so that $(j_i)$ itself belongs to $\ff(x)$, and then it is the lexicographically maximal element of $\ff(x)$ by construction.

The non-decreasingness of the map $x\mapsto b(x)$ follows from the observation that $b(x)$ is a subexpansion of every $y>x$, and hence $b(x)\le b(y)$.

If $x\ge\Lambda$, then every sequence $(j_i)$ is a subexpansion of $x$.
Hence $b(x)=M^{\infty}$, because
 $M^{\infty}$ is the lexicographically largest sequence.
\medskip

(ii) Since $j^{\infty}$ is an \emph{infinite} sequence, we may repeat  the proof of (i) by considering  infinite subexpansions everywhere.
\medskip

(iii) and (iv) We may repeat the proofs of (i) and (ii) by writing  $\Lambda=\pi(j^{\infty})$, changing the word subexpansion to superexpansion, and by changing the sense of the inequalities.
\end{proof}

We may also deduce parts (iii) and (iv) of Lemma \ref{l22} from its parts (i) and (ii) by using a \emph{reflection principle}.
Let us associate with the \emph{primary system}
\begin{equation*}
S:=\set{(d_0,q_0),\ldots,(d_M,q_M)}
\end{equation*}
the \emph{dual system}
\begin{equation*}
S':=\set{(d_0',q_0'),\ldots,(d_M',q_M')}
\end{equation*}
with
\begin{equation}\label{26}
d_j':=-d_{M-j},\quad
q_j':=q_{M-j},\quad j=0,\ldots,M,
\end{equation}
We define $\pi'$, $J'$, $\lambda'$ and $\Lambda'$ accordingly:
\begin{align*}
&\pi'((j_i))
=\pi'(j_1j_2\cdots)
:=\sum_{i=1}^{\infty}\frac{d_{j_i}'}{q_{j_1}'\cdots q_{j_i}'};\\
&J':=\set{\pi'((j_i))\ :\ (j_i)\in\set{0,\ldots,M}^{\NN}};\\
&\pi'(j_1\cdots j_m)
:=\sum_{i=1}^m\frac{d_{j_i}'}{q_{j_1}'\cdots q_{j_i}'};\\
&\lambda':=\min_j\frac{d_j'}{q_j'-1}
\qtq{and}
\Lambda':=\max_j\frac{d_j'}{q_j'-1},
\end{align*}
and we denote by $b'(x)$, $a'(x)$, $l'(x)$ and $m'(x)$ the greedy and quasi-greedy subexpansions and the lazy and quasi-lazy superexpansions of $x$ with respect to the dual system.

\begin{lemma}\label{l23}\mbox{}
\begin{enumerate}[\upshape (i)]
\item The reflection $(j_i)\mapsto\overline{(j_i)}$ is a decreasing bijection of the space of sequences onto itself, and $\pi((j_i))=-\pi'\left(\overline{(j_i)}\right)$ for every sequence.
\item The following equalities hold:
\begin{equation}\label{27}
J'=-J,\quad \lambda'=-\Lambda
\qtq{and}\Lambda'=-\lambda.
\end{equation}
\item $(j_i)$ is a subexpansion (respectively superexpansion, expansion) of $x$ in $S$ if and only if $\overline{(j_i)}$ is a superexpansion (respectively subexpansion, expansion) of $-x$ in $S'$.
\item $S''=S$.
\end{enumerate}
\end{lemma}

\begin{proof}
(i) The first assertion is elementary. The identity is obtained by a straighforward computation:
\begin{equation*}
\pi((j_i))=\sum_{i=1}^\infty\frac{d_{j_i}}{q_{j_1}\cdots q_{j_i}}=\sum_{i=1}^\infty\frac{-d'_{M-j_i}}{q'_{M-j_1}\cdots q'_{M-j_i}}=-\pi'\left(\overline{(j_i)}\right).
\end{equation*}
Furthermore, (ii) and (iii) follow from (i), (iv) follows from the definition, and (v), (vi) follow from the relations \eqref{26} and \eqref{27}.
\end{proof}

Lemmas \ref{l21}--\ref{l23} hold for all alphabet-base systems.
Now we make a further hypothesis.
We call the system $S$ \emph{semi-regular} if it satisfies the first two conditions \eqref{13} and \eqref{14} of the regularity.

\begin{lemma}\label{l24}
If $S$ is semi-regular, then
the following relations hold:
\begin{align}
&\frac{d_0+x}{q_0}
<\cdots<
\frac{d_M+x}{q_M}
\qtq{for all}x\in[\lambda,\Lambda];\label{28}\\
&\frac{d_0}{q_0-1}<\frac{d_1}{q_1-1}
<\cdots<\frac{d_M}{q_M-1}\label{29}\\
&\lambda=\pi(0^{\infty})=\frac{d_0}{q_0-1}
\qtq{and}
\Lambda=\pi(M^{\infty})=\frac{d_M}{q_M-1};\label{210}\\
&\text{if}\quad 0^{\infty}<(j_i)<M^{\infty},\qtq{then}\lambda< \pi((j_i))<\Lambda.\label{211}
\end{align}
Furthermore, $S'$ is also semi-regular.
\end{lemma}

\begin{proof}
Taking convex linear combinations of the  inequalities \eqref{13} and \eqref{14}, we get \eqref{28}.
Next, setting $x=\frac{d_j}{q_j-1}$ for any fixed $j<M$, we have
\begin{equation*}
x=\frac{d_j+x}{q_j}<\frac{d_{j+1}+x}{q_{j+1}}
\end{equation*}
by \eqref{28}, and hence
\begin{equation*}
x<\frac{d_{j+1}}{q_{j+1}-1}.
\end{equation*}
This proves \eqref{29}.
The equalities \eqref{210} follow from \eqref{11}, \eqref{12} and \eqref{29}.

To prove \eqref{211}, we consider a sequence $0^{\infty}<(j_i)<M^{\infty}$.
Denoting by $n$ the smallest index satisfying $j_n\ne 0$, we have $j_n>0$, and we infer from \eqref{29} and the equality $\lambda=\pi(0^{\infty})$ that (for  $n=1$ we use the conventions $\pi(0^{n-1})=0$ and $q_0^{n-1}=1$)
\begin{align*}
\pi((j_i))
&=\pi(0^{n-1})+\frac{d_{j_n}+\pi(d_{j_{n+1}}d_{j_{n+2}}\cdots)}{q_0^{n-1}q_{j_n}}
\ge \pi(0^{n-1})+\frac{d_{j_n}+\lambda}{q_0^{n-1}q_{j_n}}\\
&>\pi(0^{n-1})+\frac{d_0+\lambda}{q_0^n}
=\pi(0^{n-1})+\frac{\lambda}{q_0^{n-1}}
=\pi(0^{\infty})
=\lambda.
\end{align*}
Similarly, denoting by $m$  the smallest index satisfying $j_m\ne M$, we have $j_m<M$, and we  infer from \eqref{29} and the equality $\Lambda=\pi(M^{\infty})$ that
\begin{align*}
\pi((j_i))
&=\pi(M^{m-1})+\frac{d_{j_m}+\pi(d_{j_{m+1}}d_{j_{m+2}}\cdots)}{q_M^{m-1}q_{j_m}}
\le \pi(M^{m-1})+\frac{d_{j_m}+\Lambda}{q_M^{m-1}q_{j_m}}\\
&<\pi(M^{m-1})+\frac{d_M+\Lambda}{q_M^m}
=\pi(M^{m-1})+\frac{\Lambda}{q_M^{m-1}}
=\pi(M^{\infty})
=\Lambda.
\end{align*}
Finally, the semi-regularity of $S'$ follows from the relations \eqref{26} and \eqref{27}.
\end{proof}

In the following proposition we show that the classical greedy, quasi-greedy, lazy and quasi-lazy \emph{expansions} may be extended to all \emph{regular} alphabet--base systems.

\begin{proposition}\label{p25}
Let $S$ be a regular alphabet--base system.
Then $S'$ is also regular, and the greedy subexpansions, the quasi-greedy subexpansions, the lazy superexpansions, the quasi-lazy superexpansions are in fact \emph{expansions} on $[\lambda ,\Lambda ]$.
More precisely:
\color{black}
\begin{enumerate}[\upshape (i)]
\item The greedy subexpansions have the following properties:
\begin{enumerate}[\upshape (a)]
\item $\pi(b(x))=x$ for all $x\in[\lambda ,\Lambda ]$.
\item The map $x\mapsto b(x)$ is strictly increasing on $[\lambda ,\Lambda ]$.
\item A sequence $(j_i)$ is greedy if and only if
\begin{equation}\label{212}
\pi(j_{n+1}j_{n+2}\cdots)
<q_{j_n}\left(\frac{d_{j_n+1}+\lambda }{q_{j_n+1}}-\frac{d_{j_n}}{q_{j_n}}\right)\qtq{whenever}j_n<M.
\end{equation}
\end{enumerate}

\item The quasi-greedy subexpansions have the following properties:
\begin{enumerate}[\upshape (a)]
\item $\pi(a(x))=x$ for all $x\in[\lambda ,\Lambda ]$.
\item The map $x\mapsto a(x)$ is strictly increasing on $[\lambda ,\Lambda ]$.
\item An infinite sequence $(j_i)$ is quasi-greedy if and only if
\begin{equation}\label{213}
\pi(j_{n+1}j_{n+2}\cdots)
\le q_{j_n}\left(\frac{d_{j_n+1}+\lambda }{q_{j_n+1}}-\frac{d_{j_n}}{q_{j_n}}\right)\qtq{whenever}j_n<M.
\end{equation}
\end{enumerate}

\item The lazy superexpansions have the following properties:
\begin{enumerate}[\upshape (a)]
\item $\pi(l(x))=x$ for all $x\in[\lambda ,\Lambda ]$.
\item The map $x\mapsto l(x)$ is strictly increasing on $[\lambda ,\Lambda ]$.
\item A sequence $(j_i)$ is lazy if and only if
\begin{equation*}\label{214}
\pi(j_{n+1}j_{n+2}\cdots)
>q_{j_n}\left(\frac{d_{j_n-1}+\Lambda }{q_{j_n-1}}-\frac{d_{j_n}}{q_{j_n}}\right)\qtq{whenever}j_n>0.
\end{equation*}
\end{enumerate}

\item The quasi-lazy superexpansions have the following properties:
\begin{enumerate}[\upshape (a)]
\item $\pi(m(x))=x$ for all $x\in[\lambda ,\Lambda ]$.
\item The map $x\mapsto m(x)$ is strictly increasing on $[\lambda ,\Lambda ]$.
\item A co-infinite sequence $(j_i)$ is quasi-lazy if and only if
\begin{equation*}\label{215}
\pi(j_{n+1}j_{n+2}\cdots)
\ge q_{j_n}\left(\frac{d_{j_n-1}+\Lambda }{q_{j_n-1}}-\frac{d_{j_n}}{q_{j_n}}\right)\qtq{whenever}j_n>0.
\end{equation*}
\end{enumerate}
\end{enumerate}
\end{proposition}

\begin{proof}
The regularity of $S'$ follows from the relations \eqref{26} and \eqref{27}.
\medskip

(ia) Since $b(x)$ is a subexpansion of $x$, we have $\pi(b(x))\le x$ for every $x\ge\lambda$.
It remains to show that $\pi(b(x))\ge x$ for every $x\in [\lambda ,\Lambda ]$.
Since $b(x)=(j_i)$ is the largest subexpansion of $x$, we have
\begin{equation}\label{216}
\pi(j_1\cdots j_{n-1}(j_n+1)0^{\infty})>x
\end{equation}
whenever $j_n<M$.
If there are infinitely many elements $j_n<M$, then letting $n\to\infty$ in this inequality and using the continuity of $\pi$ (Lemma \ref{l21} (i)) we conclude that $\pi((j_i))\ge x$.

Next we show that we cannot have a last index $j_n<M$.
Indeed, in that case  we would have \eqref{216} again, and
\begin{equation*}
\pi(j_1\cdots j_{n-1}j_nM^{\infty})
=\pi(b(x))\le x.
\end{equation*}
Hence we would infer the inequality
\begin{equation*}
\pi(j_nM^{\infty})
<\pi((j_n+1)0^{\infty}).
\end{equation*}
By \eqref{210} this is equivalent to
\begin{equation*}
\frac{d_{j_n}+\Lambda}{q_{j_n}}
<\frac{d_{j_n+1}+\lambda}{q_{j_n+1}},
\end{equation*}
and contradicts \eqref{15}.

Finally, if $b(x)=M^{\infty}$, then $\pi(M^{\infty})=\Lambda \ge x$ by the choice of $x$.
\medskip

(ib) If $\lambda\le x<y\le\Lambda$, then $b(x)\le b(y)$ by Lemma \ref{l22} (i), and $b(x)\ne b(y)$ because $\pi(b(x))=x<y=\pi(b(y))$ by (ia).
\medskip

(ic) Since $b(x)=b(\Lambda)$ for all $x>\Lambda$, it suffices to consider the sequences $b(x)$ for $x\in[\lambda,\Lambda]$.
Then $\pi(b(x))=x$ by (ia), and a sequence
 $(j_i)$ is greedy if and only if
\begin{equation*}
\pi(j_1\cdots j_{n-1}mk_{n+1}k_{n+2}\cdots)>\pi((j_i))
\end{equation*}
for all sequences $(k_{n+j})$,
whenever $j_n<m\le M$.
Equivalently,
\begin{equation*}
\frac{d_m+\pi(k_{n+1}k_{n+2}\cdots)}{q_m}
>\frac{d_{j_n}+\pi(j_{n+1}j_{n+2}\cdots)}{q_{j_n}}
\end{equation*}
for all sequences $(k_{n+j})$,
whenever $j_n<m\le M$, and this is equivalent to
\begin{equation*}
\frac{d_m+\lambda}{q_m}
>\frac{d_{j_n}+\pi(j_{n+1}j_{n+2}\cdots)}{q_{j_n}}
\qtq{whenever}j_n<m\le M
\end{equation*}
by \eqref{22}.
In view of \eqref{13} this is equivalent to
\begin{equation*}
\frac{d_{j_n+1}+\lambda}{q_{j_n+1}}
>\frac{d_{j_n}+\pi(j_{n+1}j_{n+2}\cdots)}{q_{j_n}}\qtq{whenever}j_n<M,
\end{equation*}
i.e., to \eqref{212}.
\medskip

(iia) As in (ia), it is sufficient to show that $\pi(a(x))\ge x$ for every $x\in [\lambda ,\Lambda ]$.
This is true if $a(x)=M^{\infty}$ because $\pi(M^{\infty})=\Lambda \ge x$ by \eqref{210}.

Since $a(x)=(j_i)$ is the largest infinite subexpansion of $x$, we have
\begin{equation*}
\pi(j_1\cdots j_{n-1}(j_n+1)M^{\infty})>x
\end{equation*}
whenever $j_n<M$.
If there are infinitely many elements $j_n<M$, then letting $n\to\infty$ in this inequality and using the continuity of $\pi$ we conclude that $\pi((j_i))\ge x$.

If $(j_i)$ has a last element $j_n<M$, then we have
\begin{equation*}
\pi(a(x))=\pi(j_1\cdots j_{n-1}j_nM^{\infty})
\le x<\pi(j_1\cdots j_{n-1}(j_n+1)0^kM^{\infty})
\end{equation*}
for all $k\ge 1$ because $a(x)$ is the largest infinite subexpansion of $x$.
Letting $k\to\infty$ this yields
\begin{equation}\label{217}
\pi(a(x))=\pi(j_1\cdots j_{n-1}j_nM^{\infty})
\le
x\le\pi(j_1\cdots j_{n-1}(j_n+1)0^{\infty}).
\end{equation}
We conclude the proof by observing that the  two inequalities in \eqref{217} are in fact equalities.
Indeed, we have
\begin{equation*}
\frac{d_{j_n+1}+\lambda}{q_{j_n+1}}
\le
\frac{d_{j_n}+\Lambda}{q_{j_n}}
\end{equation*}
by our assumption \eqref{15}, and this is equivalent to
\begin{equation*}
\pi(j_1\cdots j_{n-1}(j_n+1)0^{\infty})\le \pi(j_1\cdots j_{n-1}j_nM^{\infty}).
\end{equation*}
\medskip

(iib) If $\lambda\le x<y\le\Lambda$, then $a(x)\le a(y)$ by Lemma \ref{l22} (ii), and $a(x)\ne a(y)$ because $\pi(a(x))=x<y=\pi(a(y))$ by (iia).
\medskip

(iic) Since $a(x)=a(\Lambda)$ for all $x>\Lambda$, it suffices to consider the sequences $a(x)$ for $x\in[\lambda,\Lambda]$.
Then $\pi(a(x))=x$ by (iia), and an infinite sequence $(j_i)$ is quasi-greedy if and only if
\begin{equation*}
\pi(j_1\cdots j_{n-1}mk_{n+1}k_{n+2}\cdots)>\pi((j_i))
\end{equation*}
for all  sequences $(k_{n+j})$ with infinitely many non-zero digits,whenever $j_n<m\le M$.
(We had to be careful here in order to exclude the infinite sequence $0^{\infty}$.)
Equivalently,
\begin{equation*}
\frac{d_m+\pi(k_{n+1}k_{n+2}\cdots)}{q_m}
>\frac{d_{j_n}+\pi(j_{n+1}j_{n+2}\cdots)}{q_{j_n}}
\end{equation*}
for all  sequences $(k_{n+j})$ with infinitely many non-zero digits,
whenever $j_n<m\le M$, and this is equivalent to
\begin{equation*}
\frac{d_m+\lambda}{q_m}
\ge\frac{d_{j_n}+\pi(j_{n+1}j_{n+2}\cdots)}{q_{j_n}}
\qtq{whenever}j_n<m\le M
\end{equation*}
by \eqref{22}.
In view of \eqref{13} this is equivalent to
\begin{equation*}
\frac{d_{j_n+1}+\lambda}{q_{j_n+1}}
\ge\frac{d_{j_n}+\pi(j_{n+1}j_{n+2}\cdots)}{q_{j_n}}\qtq{whenever}j_n<M,
\end{equation*}
i.e., to \eqref{213}.
\medskip

Since $S'$ is regular, the statements (iii) and (iv) follow from (i) and (ii) by applying Lemma \ref{l23} (iii).
\end{proof}

\begin{proof}[Proof of Theorem \ref{t11}]
Part (i) follows from \eqref{22}.
Part (ii) follows from Lemma \ref{l22} and the statements (a) of Proposition \ref{p25}.
\end{proof}

\section{Proof of Theorems \ref{t12} and \ref{t15}}\label{s3}

We start by checking that the sequences $\alpha^j$ and $\gamma^j$ in Theorem \ref{t12} are well defined:

\begin{lemma}\label{l31}
If $S$ is  a regular alphabet--base system, then
\begin{equation*}
\lambda<
q_j\left(\frac{d_{j+1}+\lambda}{q_{j-1}}-\frac{d_j}{q_j}\right)
\le\Lambda,\quad j=0,\ldots,M-1,
\end{equation*}
and
\begin{equation*}
\lambda\le
q_j\left(\frac{d_{j-1}+\Lambda}{q_{j-1}}-\frac{d_j}{q_j}\right)
<\Lambda,\quad
j=1,\ldots,M,
\end{equation*}
so that the corresponding expansions
\begin{equation}\label{31}
\alpha^j
=a\left(q_j\left(\frac{d_{j+1}+\lambda}{q_{j+1}}-\frac{d_j}{q_j}\right)\right)
\qtq{and}
\gamma^j
=m\left(
q_j\left(\frac{d_{j-1}+\Lambda}{q_{j-1}}-\frac{d_j}{q_j}\right)
\right)
\end{equation}
are well defined by Theorem \ref{t11}.
\end{lemma}

\begin{proof}
The stated inequalities follow from the definition \eqref{13}--\eqref{15} of the regularity.
\end{proof}

Now we may clarify the relations between greedy and quasi-greedy expansions, and between lazy and quasi-lazy expansions:

\begin{proposition}\label{p32}
Assume that $S$ is  a regular alphabet--base system, and let $x\in[\lambda,\Lambda]$.

\begin{enumerate}[\upshape (i)]
\item If $b(x)$ is infinite, then $a(x)=b(x)$.
Otherwise, $b(x)=j_1\cdots j_n0^{\infty}$  with  $j_n>0$, and then $a(x)=j_1\cdots j_{n-1}(j_n-1)k_1k_2\cdots$ with
\begin{equation}\label{32}
k_1k_2\cdots=a\left(q_{j_n-1}\left(\frac{d_{j_n}+\lambda}{q_{j_n}}-\frac{d_{j_n-1}}{q_{j_n-1}}\right)\right).
\end{equation}
\item If $l(x)$ is co-infinite, then $m(x)=l(x)$.
Otherwise, $l(x)=j_1\cdots j_n M^\infty$ with $j_n<M$, and then $m(x)=j_1\cdots j_{n-1}(j_n+1)k_1k_2\cdots$ with
\begin{equation*}\label{33}
k_1k_2\cdots=m\left(
q_{j_n+1}\left(\frac{d_{j_n}+\Lambda}{q_{j_n}}-\frac{d_{j_n+1}}{q_{j_n+1}}\right)
\right).
\end{equation*}
\end{enumerate}
\end{proposition}

\begin{proof}
(i) If $b(x)$ is infinite, then it is also the largest \emph{infinite} expansion of $x$, so that $a(x)=b(x)$.
Now let $b(x)=j_1\cdots j_n0^{\infty}$
with  $j_n>0$, and set $c(x)=j_1\cdots j_{n-1}(j_n-1)k_1k_2\cdots.$
First of all, $c(x)$ is an expansion of $x$ because
\begin{align*}
x=\pi(j_1\cdots j_n 0^\infty)&=\pi(j_1\cdots j_{n-1})+\frac{d_{j_n}+\lambda}{q_{j_1}\cdots q_{j_n}}\\
&=\pi(j_1\cdots j_{n-1})+\frac{d_{j_n-1}}{q_1\cdots q_{j_{n-1}}q_{j_n-1}}+\left(\frac{d_{j_n}+\lambda}{q_{j_1}\cdots q_{j_n}}-\frac{d_{j_n-1}}{q_1\cdots q_{j_{n-1}}q_{j_n-1}}\right)\\
&=\pi(j_1\cdots j_{n-1})+\frac{d_{j_n-1}+q_{j_n-1}\left(\frac{d_{j_n}+\lambda}{q_{j_n}}-\frac{d_{j_n-1}}{q_{j_n-1}}\right)}{q_1\cdots q_{j_{n-1}}q_{j_n-1}}\\
&=\pi(j_1\cdots j_{n-1})+\frac{d_{j_n-1}+\pi(k_1k_2\cdots)}{q_1\cdots q_{j_{n-1}}q_{j_n-1}}
\\
&=\pi(j_1\cdots j_{n-1}(j_n-1)k_1k_2\cdots).
\end{align*}
Since $(k_i)$ is an infinite expansion by \eqref{32}, and
\begin{equation*}
\pi(k_1k_2\cdots)=q_{j_n-1}\left(\frac{d_{j_n}+\lambda}{q_{j_n}}-\frac{d_{j_n-1}}{q_{j_n-1}}\right)>\lambda
\end{equation*}
by \eqref{13}, $(k_i)\ne 0^{\infty}$, and hence $(k_i)$ has infinitely many nonzero elements.
Consequently, $c(x)$ is an \emph{infinite} expansion of $x$, and therefore $c(x)\preceq a(x)$.

For the proof of the converse inequality $a(x)\preceq c(x)$ we write $a(x)=(a_i)$. We deduce from the relation $a(x)\prec b(x)$ that $a_1\cdots a_n\prec j_1\cdots j_n$, or equivalently $a_1\cdots a_n\preceq j_1\cdots j_{n-1}(j_n-1)$.
If this inequality is strict, then $a(x)\prec c(x)$.
If we have an equality, then the equality
\begin{equation*}
\pi(a(x))=\pi(c(x)),
\end{equation*}
i.e.,
\begin{equation*}
\pi(a_1a_2\cdots)=\pi(j_1\cdots j_{n-1}(j_n-1)k_1k_2\cdots)
\end{equation*}
reduces to
\begin{equation*}
\pi(a_{n+1}a_{n+2}\cdots)=\pi(k_1k_2\cdots).
\end{equation*}
Since $k_1k_2\cdots$ is a quasi-greedy sequence and $a_{n+1}a_{n+2}\cdots$ is an infinite sequence, we conclude that $a_{n+1}a_{n+2}\cdots\preceq k_1k_2\cdots,$ and the inequality $a(x)\preceq c(x)$ follows:
\begin{equation*}
a(x)=j_1\cdots j_{n-1}(j_n-1)a_{n+1}a_{n+2}\cdots
\preceq j_1\cdots j_{n-1}(j_n-1)k_1k_2\cdots
=c(x).
\end{equation*}
\medskip

(ii) Since $S'$ is regular by Proposition \ref{p25}, this follows from (i) by reflection, using Lemma \ref{l23} (iii).
\end{proof}

The following proposition is crucial for the proof of Theorem \ref{t12}.
The proof is based on a generalization  of a construction first given in \cite{EJK1990}, and later extended in \cite{BK2007}, with some delicate  modifications due to the  possibility of multiple bases.
We use the notation \eqref{31} of Lemma \ref{l31}, and we write $\alpha^j=(\alpha^j_i)$.
We will also use the notation
\begin{equation*}
\pi(j_1\cdots j_m)
:=\sum_{i=1}^m\frac{d_{j_i}}{q_{j_1}\cdots q_{j_i}}
\end{equation*}
for non-empty finite words, and the identities
\begin{equation*}
\pi(j_1j_2\cdots)
=\pi(j_1\cdots j_m)+\frac{\pi(j_{m+1}j_{m+2}\cdots)}{q_{j_1}\cdots q_{j_m}},\quad m=1,2,\ldots .
\end{equation*}

\begin{proposition}\label{p33}
Let $S$ be a regular system.
\begin{enumerate}[\upshape (i)]
\item
A sequence $(j_i)$ is greedy if and only if
\begin{equation}\label{34}
(j_{n+i})\prec \alpha^{j_n}\qtq{whenever}j_n<M.
\end{equation}

\item An infinite sequence $(j_i)$ is quasi-greedy if and only if
\begin{equation}\label{35}
(j_{n+i})\preceq \alpha^{j_n}\qtq{whenever}j_n<M.
\end{equation}

\item A sequence $(j_i)$ is lazy if and only if
\begin{equation*}\label{36}
(j_{n+i})\succ\gamma^{j_n}
\qtq{whenever}j_n>0.
\end{equation*}

\item A co-infinite sequence $(j_i)$ is quasi-lazy if and only if
\begin{equation*}\label{37}
(j_{n+i})\succeq\gamma^{j_n}
\qtq{whenever}j_n>0.
\end{equation*}
\end{enumerate}
\end{proposition}

\begin{proof}
(i) Fix an arbitrary sequence $(j_i)$.
By Proposition \ref{p25} we have to show that the condition \eqref{34} is equivalent to the following property:
\begin{equation}\label{38}
\pi(j_{n+1}j_{n+2}\cdots)
<q_{j_n}\left(\frac{d_{j_n+1}+\lambda }{q_{j_n+1}}-\frac{d_{j_n}}{q_{j_n}}\right)\qtq{whenever}j_n<M.
\end{equation}

Let us denote by $x<y$ the two sides of this inequality.
First we assume \eqref{38}, and we consider an index $n$ such that $j_n<M$.
Then $(j_{n+i})$ is an expansion of $x$, and hence by \eqref{38} it is also a subexpansion of $y$.
If the sequence $(j_{n+i})$ is infinite, then, since  $\alpha^{j_n}$ is the largest  subexpansion of $y$ by definition, hence we infer that
$(j_{n+i})\preceq \alpha^{j_n}$.
This implies \eqref{34} because
$\pi((j_{n+i}))=x<y=\pi(\alpha^{j_n})$ and therefore $(j_{n+i})\ne \alpha^{j_n}$.

If the sequence $(j_{n+i})$ has a last element $j_m>0$ ($m>n$), then we consider the number
\begin{equation*}
z:=\min\set{q_{j_{n+1}}\cdots q_{j_m}(y-\pi(j_{n+1}\cdots j_m)),\Lambda}.
\end{equation*}
We have
\begin{equation*}
y>x=\pi(j_{n+1}\cdots j_m)+\frac{\pi(0^{\infty})}{q_{j_{n+1}}\cdots q_{j_m}}
=\pi(j_{n+1}\cdots j_m)+\frac{\lambda}{q_{j_{n+1}}\cdots q_{j_m}},
\end{equation*}
and therefore $\lambda<z\le\Lambda$, so that the quasi-greedy expansion $k_1k_2\cdots$ of $z$ has infinitely many nonzero elements.
Since
\begin{equation*}
\pi(j_1\cdots j_mk_1k_2\cdots)
=\pi(j_1\cdots j_m)+\frac{\pi(k_1k_2\cdots)}{q_{j_{n+1}}\cdots q_{j_m}}
=\pi(j_1\cdots j_m)+\frac{z}{q_{j_{n+1}}\cdots q_{j_m}}
\le y,
\end{equation*}
it follows that $j_1\cdots j_mk_1k_2\cdots$ is an infinite subexpansion of $y$, and hence
it is lexicographically smaller than or equal to $a(y)=\alpha^{j_n}$.
Therefore
\begin{equation*}
(j_i)
<j_1\cdots j_mk_1k_2\cdots
\le \alpha^{j_n},
\end{equation*}
proving \eqref{34} again.

Now we assume \eqref{34}, and we consider an index $n$ such that $j_n<M$.
We have to prove the inequality \eqref{38}.
Fix an $n\ge 1$ satisfying $j_n<M$.
Starting with $k_0:=n$, we construct a sequence $k_0<k_1<k_2<\cdots$ of integers as follows.
If $k_p$ has already been defined for some $p\ge 0$ and $j_{k_p}<M$, then
 $(j_{k_p+i})<\alpha^{j_{k_p}}$ by \eqref{34}.
Therefore there exists an integer $k_{p+1}>k_p$ such that
\begin{equation*}
j_{k_p+1}\cdots j_{k_{p+1}-1}=\alpha^{j_{k_p}}_1\cdots \alpha^{j_{k_p}}_{k_{p+1}-k_p-1}
\qtq{and}
j_{k_{p+1}}<\alpha^{j_{k_p}}_{k_{p+1}-k_p}.
\end{equation*}
Now \eqref{38} is obtained by the following computation, where we use for brevity the convention $q_{j_{k_0+1}}\cdots q_{j_{k_0}}:=1$:\begin{equation}\label{39}
\begin{split}
\pi&(j_{n+1}j_{n+2}\cdots)\\
&=\sum_{p=0}^{\infty}\frac{\pi(j_{k_p+1}\cdots j_{k_{p+1}})}{q_{j_{k_0+1}}\cdots q_{j_{k_p}}}\\
&=\sum_{p=0}^{\infty}\frac{\pi(\alpha^{j_{k_p}}_1\cdots \alpha^{j_{k_p}}_{k_{p+1}-k_p-1}j_{k_{p+1}})}{q_{j_{k_0+1}}\cdots q_{j_{k_p}}}\\
&=\sum_{p=0}^{\infty}\frac{1}{q_{j_{k_0+1}}\cdots q_{j_{k_p}}}\left(\pi(\alpha^{j_{k_p}}_1\cdots \alpha^{j_{k_p}}_{k_{p+1}-k_p-1}(j_{k_{p+1}}+1)0^{\infty})\right.\\
&\hspace{5cm}\left.-\frac{q_{j_{k_{p+1}}}}{q_{j_{k_p+1}}\cdots q_{j_{k_{p+1}}}}\left(\frac{d_{(j_{k_{p+1}}+1)}+\lambda}{q_{(j_{k_{p+1}}+1)}}-\frac{d_{j_{k_{p+1}}}}{q_{j_{k_{p+1}}}}\right)\right)\\
&<\sum_{p=0}^{\infty}\frac{1}{q_{j_{k_0+1}}\cdots q_{j_{k_p}}}\left(\pi(\alpha^{j_{k_p}})-\frac{q_{j_{k_{p+1}}}}{q_{j_{k_p+1}}\cdots q_{j_{k_{p+1}}}}\left(\frac{d_{(j_{k_{p+1}}+1)}+\lambda}{q_{(j_{k_{p+1}}+1)}}-\frac{d_{j_{k_{p+1}}}}{q_{j_{k_{p+1}}}}\right)\right)\\
&=\sum_{p=0}^{\infty}\frac{q_{j_{k_p}}}{q_{j_{k_0+1}}\cdots q_{j_{k_p}}}\left(\frac{d_{j_{k_p}+1}+\lambda}{q_{j_{k_p}+1}}-\frac{d_{j_{k_p}}}{q_{j_{k_p}}}\right)\\
&\hspace{5cm}-\frac{q_{j_{k_{p+1}}}}{q_{j_{k_0+1}}\cdots q_{j_{k_{p+1}}}}\left(\frac{d_{{(j_{k_{p+1}}+1)}}+\lambda}{q_{(j_{k_{p+1}}+1)}}-\frac{d_{j_{k_{p}+1}}}{q_{j_{k_{p+1}}}}\right)\\
&=q_{j_{k_0}}\left(\frac{d_{j_{k_0}+1}+\lambda}{q_{j_{k_0}+1}}-\frac{d_{j_{k_0}}}{q_{j_{k_0}}}\right).
\end{split}
\end{equation}
Let us explain the third equality.
It is sufficient to show that
\begin{multline*}
\pi(\alpha^{j_{k_p}}_1\cdots \alpha^{j_{k_p}}_{k_{p+1}-k_p-1}j_{k_{p+1}})
=\pi(\alpha^{j_{k_p}}_1\cdots \alpha^{j_{k_p}}_{k_{p+1}-k_p-1}(j_{k_{p+1}}+1)0^{\infty})\\
-\frac{q_{j_{k_{p+1}}}}{q_{j_{k_p+1}}\cdots q_{j_{k_{p+1}}}}\left(\frac{d_{(j_{k_{p+1}}+1)}+\lambda}{q_{(j_{k_{p+1}}+1)}}-\frac{d_{j_{k_{p+1}}}}{q_{j_{k_{p+1}}}}\right)
\end{multline*}
for every $p\ge 0$.
Subtracting $\pi(\alpha^{j_{k_p}}_1\cdots \alpha^{j_{k_p}}_{k_{p+1}-k_p-1})$ and  multiplying by $q_{j_{k_p+1}}\cdots q_{j_{k_{p+1}-1}}$, this is equivalent to
\begin{equation*}
\frac{d_{j_{k_{p+1}}}}{q_{j_{k_{p+1}}}}
=\frac{d_{(j_{k_{p+1}}+1)}+\pi(0^{\infty})}{q_{(j_{k_{p+1}}+1)}}
-\left(\frac{d_{(j_{k_{p+1}}+1)}+\lambda}{q_{(j_{k_{p+1}}+1)}}-\frac{d_{j_{k_{p+1}}}}{q_{j_{k_{p+1}}}}\right),
\end{equation*}
and the  latter relation holds because $\pi(0^{\infty})=\lambda$.
The strict inequality in \eqref{39} follows from \eqref{211} implying that
\begin{equation*}
\pi(\alpha^{j_{k_p}}_1\cdots \alpha^{j_{k_p}}_{k_{p+1}-k_p-1}(j_{k_{p+1}}+1)0^{\infty})
\le \pi(\alpha^{j_{k_p}}_1\cdots \alpha^{j_{k_p}}_{k_{p+1}-k_p}0^{\infty})
<\pi(\alpha^{j_{k_p}}).
\end{equation*}
\medskip

(ii) Fix an arbitrary \emph{infinite} sequence $(j_i)$.
By Proposition \ref{p25} we have to show that the condition \eqref{35} is equivalent to the following property:
\begin{equation}\label{310}
\pi(j_{n+1}j_{n+2}\cdots)
\le q_{j_n}\left(\frac{d_{j_n+1}+\lambda }{q_{j_n+1}}-\frac{d_{j_n}}{q_{j_n}}\right)\qtq{whenever}j_n<M.
\end{equation}
We denote again by $x<y$ the two sides of this inequality.
First we assume \eqref{310}.
If $j_n<M$ for some $n\ge 1$, then $(j_{n+i})$ is an infinite expansion of $x$, and therefore also an infinite subexpansion of $y$.
Since by definition $\alpha^{j_n}$ is the largest infinite subexpansion of $y$, the relation \eqref{35} follows.

Now let $(j_i)$ be an infinite sequence satisfying the lexicographic conditions \eqref{35}.
We have to prove the inequalities  \eqref{310}.
Fix an $n\ge 1$ satisfying $j_n<M$.
Starting with $k_0:=n$, we construct a finite or infinite integer sequence $k_0<k_1<k_2<\cdots$ as follows.
If $k_p$ has already been defined for some $p\ge 0$ and $j_{k_p}<M$, then
 $(j_{k_p+i})\le \alpha^{j_{k_p}}$ by \eqref{35}.
If we have an equality, then we stop the construction.
Otherwise there exists an integer $k_{p+1}>k_p$ such that
\begin{equation*}
j_{k_p+1}\cdots j_{k_{p+1}-1}=\alpha^{j_{k_p}}_1\cdots \alpha^{j_{k_p}}_{k_{p+1}-k_p-1}
\qtq{and}
j_{k_{p+1}}<\alpha^{j_{k_p}}_{k_{p+1}-k_p}.
\end{equation*}
If we obtain an infinite sequence, then \eqref{310} is obtained by the same computation \eqref{38} as in (i).

The computation is similar but simpler if the construction of the sequence $(k_p)$ is stopped after some $k_m$.
Indeed, in case $m=0$ we  have an equality in \eqref{310}.
If $m\ge 1$, then we  modify the computation \eqref{39} as follows:
\begin{align*}
\pi&(j_{n+1}j_{n+2}\cdots)\\
&=\sum_{p=0}^{m-1}\frac{\pi(j_{k_p+1}\cdots j_{k_{p+1}})}{q_{j_{k_0+1}}\cdots q_{j_{k_p}}}
+\frac{\pi(\alpha^{j_{k_m}})}{q_{j_{k_0+1}}\cdots q_{j_{k_m}}}\\
&<\sum_{p=0}^{m-1}\frac{q_{j_{k_p}}}{q_{j_{k_0+1}}\cdots q_{j_{k_p}}}\left(\frac{d_{j_{k_p}+1}+\lambda}{q_{j_{k_p}+1}}-\frac{d_{j_{k_p}}}{q_{j_{k_p}}}\right)
-\frac{q_{j_{k_{p+1}}}}{q_{j_{k_0+1}}\cdots q_{j_{k_{p+1}}}}\left(\frac{d_{{(j_{k_{p+1}}+1)}}+\lambda}{q_{(j_{k_{p+1}}+1)}}-\frac{d_{j_{k_{p}+1}}}{q_{j_{k_{p+1}}}}\right)\\
&\qquad\qquad
+\frac{q_{j_{k_m}}}{q_{j_{k_0+1}}\cdots q_{j_{k_m}}}\left(\frac{d_{j_{k_m}+1}+\lambda}{q_{j_{k_m}+1}}-\frac{d_{j_{k_m}}}{q_{j_{k_m}}}\right)\\
&=q_{j_{k_0}}\left(\frac{d_{j_{k_0}+1}+\lambda}{q_{j_{k_0}+1}}-\frac{d_{j_{k_0}}}{q_{j_{k_0}}}\right).
\end{align*}
\medskip

Since $S'$ is regular by Proposition \ref{p25}, (iii) and (iv) follow from (i) and (ii) by reflection, using Lemma \ref{l23} (iii).
\color{black}
\end{proof}

\begin{proof}[Proof of Theorem \ref{t12}]
The theorem follows from Propositions \ref{p25} and \ref{p33}.
\end{proof}

\begin{proof}[Proof of Theorem \ref{t15}]
(i) We know that $\uu'_{S}$ contains $0^{\infty}$ and $M^{\infty}$.
On the other hand,
\begin{equation*}
\uu'_{S}\subseteq \uu'_{\ss_{M,q_{GR}}}
=\set{0^{\infty},M^{\infty}}
\end{equation*}
by Corollary \ref{c13} and Example \ref{e14}.
\medskip

(ii) It follows from our assumptions that
$0^m\alpha_{GR}\in\uu'_S$ for all $m\in\NN$.
\medskip

(iii) Since $\alpha_{KL}$ is also the \emph{greedy} expansion of $1$ in base $q_{KL}$, $\alpha_{M,q}$ converges component-wise to $\alpha_{KL}$ as $q\to q_{KL}$ by \cite[Proposition 2.6]{DK2010}.
Therefore there exists a base $q<q_{KL}$ such that
\begin{equation*}
\alpha^j\prec \alpha_{M,q}\prec \alpha^{KL}
\qtq{and}
\gamma^{M-j}\succ\overline{\alpha_{M,q}}\succ\overline{\alpha^{GR}}
\end{equation*}
for all $j<M$.
Then
\begin{equation*}
\uu_S'\subseteq\uu_{M,q}'
\end{equation*}
by Corollary \ref{c13}, and we conclude by observing that $\uu_{M,q}'$ is countable by Example \ref{e14}.
\medskip

(iv) It follows from our assumptions that
\begin{equation*}
\uu'_{S}\supseteq \uu'_{\ss_{M,q_{KL}}}.
\end{equation*}
We conclude by observing that the latter set has the power of continuum by Example \ref{e14}, and that the set of all sequences also has the power of continuum.
\end{proof}

\section{Proof of Theorem \ref{t16}}\label{s4}

We start with three lemmas.

\begin{lemma}\label{l41}
An  alphabet--base system  $S=\set{(d_0,q_0),(d_1,q_1)}$ is regular if and only if
\begin{equation}\label{41}
\frac{d_0}{q_0-1}<\frac{d_1}{q_1-1}
\qtq{and}
q_0\left(d_1+\frac{d_0}{q_0-1}\right)
\le q_1\left(d_0+\frac{d_1}{q_1-1}\right).
\end{equation}
In particular, $\set{(0,q_0),(1,q_1)}$ is regular for all $q_0, q_1\in(1,2]$.
\end{lemma}

\begin{proof}
We have to check the properties \eqref{13}, \eqref{14} and \eqref{15}.
Since
\begin{align*}
&\frac{d_0+\frac{d_0}{q_0-1}}{q_0}
<\frac{d_1+\frac{d_0}{q_0-1}}{q_1}
\Longleftrightarrow
(q_1-1)\frac{d_0}{q_0-1}<d_1
\intertext{and}
&\frac{d_0+\frac{d_1}{q_1-1}}{q_0}
<\frac{d_1+\frac{d_1}{q_1-1}}{q_1}
\Longleftrightarrow
d_0<(q_0-1)\frac{d_1}{q_1-1},
\end{align*}
the conditions \eqref{13}, \eqref{14} are equivalent to the first inequality in \eqref{41}, and
\begin{equation*}
\lambda=\frac{d_0}{q_0-1}\qtq{and}
\Lambda=\frac{d_1}{q_1-1}.
\end{equation*}
Hence
\begin{equation*}
\frac{d_1+\lambda}{q_1}\le\frac{d_0+\Lambda}{q_0}
\Longleftrightarrow
\frac{d_1+\frac{d_0}{q_0-1}}{q_1}\le\frac{d_0+\frac{d_1}{q_1-1}}{q_0},
\end{equation*}
and therefore the condition \eqref{15} is equivalent to the second inequality in \eqref{41}.
\end{proof}

Henceforth we assume the regularity of  $S=\set{(d_0,q_0),(d_1,q_1)}$, and we  use the sequences
\begin{equation}\label{42}
\alpha^0=a\left(\frac{q_0}{q_1}\left(d_1+\frac{d_0}{q_0-1}\right)-d_0\right)
\qtq{and}
\gamma^1=m\left(\frac{q_1}{q_0}\left(d_0+\frac{d_1}{q_1-1}\right)-d_1\right)
\end{equation}
coming from \eqref{31}.

\begin{lemma}\label{l42}
Let $S=\set{(d_0,q_0),(d_1,q_1)}$ be a regular alphabet-base system.
Then:
\begin{enumerate}[\upshape (i)]
\item We have the equivalences
\begin{equation*}
\alpha^0\preceq(10)^\infty\Longleftrightarrow
\pi(\alpha^0)\le \pi(110^\infty)
\Longleftrightarrow
q_0\le 1+\frac{1}{q_1}.
\end{equation*}
\item If $q_1\le 2$, then the following implication holds:
\begin{equation*}
\alpha^0\preceq (10)^\infty\Longrightarrow \gamma^1\succeq 0(01)^\infty.
\end{equation*}
\item We have the equivalences
\begin{equation*}
\gamma^1\succeq(01)^\infty\Longleftrightarrow
\pi(\gamma^1)\ge \pi(001^\infty)
\Longleftrightarrow
q_1\le 1+\frac{1}{q_0}.
\end{equation*}
\item If $q_0\le 2$, then the following implication holds:
\begin{equation*}
\gamma^1\succeq(01)^\infty\Longrightarrow \alpha^0\preceq 1(10)^\infty.
\end{equation*}
\end{enumerate}
\end{lemma}

\begin{proof}
(i) If $ \pi(\alpha^0)\le \pi(110^\infty)$, since $10\alpha^0$ is the   quasi-greedy expansion of $\pi(110^\infty)$ by Proposition \ref{p32} (i), we infer from Proposition \ref{p25} (ii) that $10\alpha^0\succeq\alpha^0$  and this implies $\alpha^0\preceq(10)^\infty$.

On the other hand, if $ \pi(\alpha^0)>\pi(110^\infty)$, then the greedy expansion of $\pi(\alpha^0)$ starts with $110^j1$ for some nonnegative integer $j$.
By  Proposition \ref{p32} (i) this implies that $\alpha^0$
starts with $11$, and therefore $\alpha^0\succ (10)^\infty$.
This finishes the proof of the first equivalence.

Next we obtain by a direct computation the following equality:
\begin{align*}
\pi(110^\infty)-\pi(\alpha^0)
&=d_1\left(\frac{1}{q_1}+\frac{1}{q_1^2}\right)+\frac{d_0}{q_1^2(q_0-1)}
-\frac{q_0}{q_1}\left(d_1+\frac{d_0}{q_0-1}\right)+d_0\\
&=\left(\frac{d_1}{q_1-1}-\frac{d_0}{q_0-1}\right)
\cdot\frac{q_1-1}{q_1}\cdot
\left(1+\frac{1}{q_1}-q_0\right).
\end{align*}
The second equivalence hence follows by observing that the sign of the last expression is equal to that of its last factor by the regularity of $S$.
\medskip

(ii) Assume that $\alpha^0\preceq (10)^\infty$.
If $\gamma^1\succ (01)^\infty$, then the implication holds.
Henceforth we assume that $\gamma^1\preceq (01)^\infty$.
Then $0(01)^\infty$ is a quasi-lazy sequence by Theorem \ref{t12} (iv), and therefore
\begin{equation*}
\gamma^1\succeq 0(01)^\infty
\Longleftrightarrow \pi(\gamma^1)\ge \pi(0(01)^\infty).
\end{equation*}
In view of this equivalence it suffices to show the second inequality.
We have
\begin{align*}
\pi(\gamma^1)-\pi(0(01)^\infty)
&=\frac{q_1}{q_0}\left(d_0+\frac{d_1}{q_1-1}\right)-d_1
-\left(\frac{d_0}{q_0}+\frac{d_0q_1+d_1}{q_0(q_0q_1-1)}\right)\\
&=\left(\frac{d_1}{q_1-1}-\frac{d_0}{q_0-1}\right)
\cdot\frac{(q_0-1)q_1(q_1-1)}{q_0(q_0q_1-1)}\cdot
\left(\frac{1}{q_1}+\frac{1}{q_1-1}-q_0\right),
\end{align*}
by a direct computation.
Since $S$ is regular, it remains to show that the last factor is nonnegative.
Using (i), this follows from our assumptions $\alpha^0\preceq (10)^\infty$ and $q_1\le 2$:
\begin{equation*}
\frac{1}{q_1}+\frac{1}{q_1-1}-q_0
\ge \frac{1}{q_1}+1-q_0
\ge 0.
\end{equation*}

\medskip
(iii) and (iv) follow from (i) and (ii) by reflection.
\end{proof}

\begin{lemma}\label{l43}
Apart from $0^{\infty}$ and $1^{\infty}$, no sequence $(c_i)\in\set{0,1}^{\NN}$ satisfies the following conditions:
\begin{align}
&(c_{n+i})\prec (10)^{\infty}\qtq{whenever}c_n=0;\label{43}\\
&(c_{n+i})\succ 0(01)^{\infty}\qtq{whenever}c_n=1.\label{44}
\end{align}
\end{lemma}

\begin{proof}
Assume on the contrary that there exists a sequence $0^{\infty}<(c_i)<1^{\infty}$ satisfying \eqref{43}--\eqref{44}, and choose a positive integer $k$ such that
$0^k<c_1\cdots c_k<1^k$.
We claim that
\begin{equation}\label{45}
0(01)^{\infty}\prec (c_{n+i})\prec (10)^{\infty}
\qtq{for all}n\ge k.
\end{equation}

Indeed, the first inequality in \eqref{45} holds by our assumption if $c_n=1$.
If $c_n=0$, then there exists an $m<n$ such that $c_m\cdots c_n=10^{n-m}$, and then
\begin{equation*}
c_{n+1}c_{n+2}\cdots
\succeq 0^{n-m}c_{n+1}c_{n+2}\cdots
=c_{m+1}c_{m+2}\cdots
\succ 0(01)^{\infty}
\end{equation*}
by \eqref{44}.
Similarly, the second inequality in \eqref{45} holds by our assumption if $c_n=0$.
If $c_n=1$, then there exists an $m<n$ such that $c_m\cdots c_n=01^{n-m}$, and then
\begin{equation*}
c_{n+1}c_{n+2}\cdots
\preceq 1^{n-m}c_{n+1}c_{n+2}\cdots
=c_{m+1}c_{m+2}\cdots
\prec (10)^{\infty}
\end{equation*}
by \eqref{43}.

It follows from \eqref{45} that $(c_{k+i})$ does not contain any bloc $11$, and that it contains at most one bloc $00$.
Hence there exists an $n\ge k$ such that $(c_{n+i})=(10)^{\infty}$.
But this contradicts the second inequality in \eqref{45}.
\end{proof}

\begin{proof}[Proof of Theorem \ref{t16}]
First we show the equivalence (i) $\Longleftrightarrow$ (ii).
Our proof will also show that either $\uu_S'$ is infinite or $\uu_S'=\set{0^{\infty},1^{\infty}}$.
\color{black}
By Corollary \ref{c13}  a sequence $(c_i)$ belongs to $\uu_S'$ if and only if it satisfies the following two conditions:
\begin{align}
&(c_{n+i})\prec\alpha^0\qtq{whenever}c_n=0,\label{46}
\intertext{and}
&(c_{n+i})\succ\gamma^1\qtq{whenever}c_n=1.\label{47}
\end{align}
If $\alpha^0\succ (10)^{\infty}$ and $\gamma^1\prec (01)^{\infty}$, then the sequences $0^k(10)^{\infty}$ belong to $\uu_S'$ for every $k\in\NN$ because they satisfy the conditions \eqref{46}--\eqref{47}.
Therefore $\uu_S'$ is an infinite set.

Next we show that $\uu_S'=\set{0^{\infty},1^{\infty}}$ in every other case.
If $\alpha^0\preceq (10)^{\infty}$, then   $\gamma^1\succeq 0(01)^\infty$ by Lemma \ref{l42} (ii).
We infer from  \eqref{46}--\eqref{47} that every  $(c_i)\in\uu_S'$ satisfies the assumptions of Lemma \ref{l43}, and hence  $(c_i)\in\set{0^{\infty},1^{\infty}}$.

Similarly, if $\gamma^1\succeq (01)^{\infty}$, then  $\alpha^0\preceq 1(10)^{\infty}$ by Lemma \ref{l42} (iv), and  we deduce from \eqref{46}--\eqref{47} that every $(c_i)\in\uu_S'$ satisfies the relations
\begin{align*}
&(c_{n+i})\prec 1(10)^{\infty}\qtq{whenever}c_n=0,
\intertext{and}
&(c_{n+i})\succ (01)^{\infty}\qtq{whenever}c_n=1.
\end{align*}
Applying Lemma \ref{l43} to the reflection of $(c_i)$ instead of $(c_i)$, we conclude again that $(c_i)\in\set{0^{\infty},1^{\infty}}$.
\medskip

The equivalence (ii) $\Longleftrightarrow$ (iii) follows from Lemma \ref{l42} (i) and (iii).
\end{proof}
For the proof of Theorem \ref{t17} we need one more lemma:

\begin{lemma}\label{l44}
If $S=\set{(d_0,q_0),(d_1,q_1)}$ is a regular alphabet-base system with $d_1\ge 0\ge d_0$, then
\begin{equation*}
q_0\le q_1
\Longrightarrow
\gamma^1\preceq\overline{\alpha^0}
\qtq{and}
q_1\le q_0
\Longrightarrow
\gamma^1\succeq\overline{\alpha^0}.
\end{equation*}
\end{lemma}

\begin{proof}
We infer from \eqref{42} that
\begin{equation*}
\pi(\alpha^0)
=\frac{q_0}{q_1}\left(d_1+\frac{d_0}{q_0-1}\right)-d_0
\end{equation*}
and
\begin{equation*}
\pi(\gamma^1)
=\frac{q_1}{q_0}\left(d_0+\frac{d_1}{q_1-1}\right)-d_1
=\frac{q_1}{q_0}\left(\frac{d_1}{q_1-1}-\pi(\alpha^0)\right)+\frac{d_0}{q_0-1}.
\end{equation*}
Writing $\alpha^0=(j_i)$ and using the equalities
\begin{equation*}
\pi(\alpha^0)
=\sum_{i=1}^\infty \frac{d_{j_i}}{q_{j_1}q_{j_2}\cdots q_{j_i}}
\qtq{and}
\pi\left(\overline{\alpha^0}\right)=\sum_{i=1}^\infty \frac{d_{\overline{j_i}}}{q_{\overline{j_1}}q_{\overline{j_2}}\cdots q_{\overline{j_i}}},
\end{equation*}
hence we deduce the following identity:
\begin{equation*}\label{48}
\pi(\gamma^1)-\pi(\overline{\alpha^0})
=\frac{q_1}{q_0}\sum_{i=1}^\infty\left(\frac{d_1}{q_1^i}
-\frac{d_{j_i}}{q_{j_1}q_{j_2}\cdots q_{j_i}}\right)
+\sum_{i=1}^\infty\left(\frac{d_0}{q_0^i}
-\frac{d_{\overline{j_i}}}{q_{\overline{j_1}}q_{\overline{j_2}}\cdots q_{\overline{j_i}}}\right).
\end{equation*}

Using this identity we prove the implication
\begin{equation}\label{49}
q_0\le q_1
\Longrightarrow
\pi(\gamma^1)\le\pi(\overline{\alpha^0}).
\end{equation}
It suffices to establish
for every $i\ge 1$ the inequality
\begin{equation}\label{410}
\frac{d_0}{q_0^i}
+\frac{q_1}{q_0}\cdot\frac{d_1}{q_1^i}
\le
\frac{d_{\overline{j_i}}}{q_{\overline{j_1}}q_{\overline{j_2}}\cdots q_{\overline{j_i}}}
+
\frac{q_1}{q_0}\cdot\frac{d_{j_i}}{q_{j_1}q_{j_2}\cdots q_{j_i}}.
\end{equation}
First we observe that $j_1=1$.
Indeed, in case $j_1=0$ we would infer from Proposition \ref{p25} (ii) that $\alpha^0=0^{\infty}$, and therefore $\pi(\alpha^0)=\lambda$, contradicting Lemma \ref{l31}.
Since $j_1=1$, \eqref{410} reduces to an equality for $i=1$, while for $i\ge 2$ it may be simplified to
\begin{equation*}
\frac{d_0}{q_0^{i-1}}
+\frac{d_1}{q_1^{i-1}}
\le
\frac{d_{\overline{j_i}}}{q_{\overline{j_2}}\cdots q_{\overline{j_i}}}
+
\frac{d_{j_i}}{q_{j_2}\cdots q_{j_i}},
\end{equation*}
and then to
\begin{equation*}
\frac{d_0}{q_0^{i-1}}
+\frac{d_1}{q_1^{i-1}}
\le
\frac{d_0}{q_0^kq_1^{i-1-k}}
+
\frac{d_1}{q_1^kq_0^{i-1-k}}
\end{equation*}
with a suitable integer $0\le k\le i-1$, depending on $i$.
(In case $j_i=0$ we exchange the two fractions on the right hand side.)
The required inequality follows by observing that
\begin{equation*}
\frac{d_0}{q_0^kq_1^{i-1-k}}
+
\frac{d_1}{q_1^kq_0^{i-1-k}}
-\frac{d_0}{q_0^{i-1}}
-\frac{d_1}{q_1^{i-1}}
=\left(\frac{1}{q_0^{i-1-k}}-\frac{1}{q_1^{i-1-k}}\right)
\left(\frac{d_1}{q_1^k}-\frac{d_0}{q_0^k}\right)\ge 0
\end{equation*}
by our assumptions $q_0\le q_1$ and $d_0\le 0\le d_1$.

Since $m(\pi(\overline{\alpha^0})$ is the smallest co-infinite expansion of $\pi(\overline{\alpha^0})$, applying Proposition \ref{p25} (iv)  we infer from \eqref{49} the implication.
\begin{equation*}
q_0\le q_1
\Longrightarrow
\gamma^1
=m(\pi(\gamma^1))
\preceq m(\pi(\overline{\alpha^0}))
\preceq \overline{\alpha^0}.
\end{equation*}
The case $q_0\ge q_1$ of the lemma hence follows by reflection.
\end{proof}

\begin{remark}
If $q_0=q_1$, then the proof of  Lemma \ref{l44}  remains valid without any assumption on $d_0$ and $d_1$.
\end{remark}

\begin{proof}[Proof of Theorem \ref{t17}]
The system is regular by Lemma \ref{l41}.
Assume that $q_0\le q_1$.
If
\begin{equation*}
q_0\le 1+\frac{1}{q_1},
\end{equation*}
then $\alpha^0\preceq (10)^\infty$ by Lemma \ref{l42} (i), and then
$\uu_S'=\set{0^{\infty},M^{\infty}}$ by Theorem \ref{t16}.
Otherwise,
\begin{equation*}
\alpha^0\succ (10)^\infty\qtq{and}\gamma^1\prec (01)^\infty
\end{equation*}
by Lemma \ref{l42} (i) and Lemma \ref{l44}, and then $\uu_S'$ is an infinite set by Theorem \ref{t16}.

The proof in case $q_0\ge q_1$ is analogous, by using Lemma \ref{l42} (iii).
\end{proof}

\end{document}